\documentclass{article}%
\usepackage{amsmath}
\usepackage{amsfonts}
\usepackage{amssymb}
\usepackage{graphicx}%
\numberwithin{equation}{section}
\newtheorem{theorem}{Theorem}[section]

\newtheorem{definition}[theorem]{Definition}

\newtheorem{proposition}[theorem]{Proposition}
\newtheorem{remark}[theorem]{Remark}

\newenvironment{proof}[1][Proof]{\noindent\textbf{#1.} }{\ \rule{0.5em}{0.5em}}
\begin{document}

\title{$C^{k,\alpha}$-regularity of solutions to \\quasilinear equations structured on \\H\"{o}rmander's vector fields\thanks{2000 AMS Classification: Primary 35H20.
Secondary: 35J62, 35B65, 42B20. Key-words: H\"{o}rmander's vector fields,
quasilinear equations, $C^{k,\alpha}$ regularity.}}
\author{Marco Bramanti, Maria Stella Fanciullo}
\maketitle

\begin{abstract}
For a linear nonvariational operator structured on smooth H\"{o}rmander's
vector fields, with H\"{o}lder continuous coefficients, we prove a regularity
result in the scale of $C_{X}^{k,\alpha}$ spaces. We deduce an analogous
regularity result for nonvariational degenerate quasilinear equations.

\end{abstract}

\section*{Introduction}

Let $X_{1},X_{2},...,X_{q}$ be a system of smooth H\"{o}rmander's vector
fields in a bounded smooth domain $\Omega$ of $\mathbb{R}^{n}$, with $q<n$
(see \S \ref{sec Hormander} for precise definitions). Nonvariational operators
of the kind%
\[
L=\sum_{i,j=1}^{q}a_{ij}\left(  x\right)  X_{i}X_{j}+\sum_{i=1}^{q}%
b_{i}\left(  x\right)  X_{i}+c\left(  x\right)  \,,
\]
with $\left\{  a_{ij}\right\}  $ real symmetric uniformly positive matrix,
have been studied by several authors, establishing in particular local a
priori estimates on $X_{i}X_{j}u$ in H\"{o}lder or $L^{p}$ spaces, in terms of
$Lu$ and $u,$ and assuming the coefficients $a_{ij}$ bounded and,
respectively, H\"{o}lder continuous or VMO: see \cite{BB2000}, \cite{BZ2} for
$L^{p}$ estimates and \cite{BB2007}, \cite{BZ2}, \cite{GL} for Schauder
estimates. In particular, in \cite{BB2007} evolution operators of the kind
\[
H=\partial_{t}-\sum_{i,j=1}^{q}a_{ij}\left(  t,x\right)  X_{i}X_{j}+\sum
_{i=1}^{q}b_{i}\left(  t,x\right)  X_{i}+c\left(  t,x\right)
\]
have been studied, and local a priori Schauder estimates of the following kind
have been proved: if $a_{ij},b_{i},c\in C^{k,\alpha}\left(  U\right)  $ for
some integer $k\geqslant0$ and some $\alpha\in\left(  0,1\right)  ,$ then for
every domain $U^{\prime}\Subset U$, $u\in C_{loc}^{k+2,\alpha}\left(
U\right)  $ with $Hu\in C^{k,\alpha}\left(  U\right)  ,$%
\[
\left\Vert u\right\Vert _{C^{k+2,\alpha}\left(  U^{\prime}\right)  }\leqslant
c\left\{  \left\Vert Hu\right\Vert _{C^{k,\alpha}\left(  U\right)
}+\left\Vert u\right\Vert _{L^{\infty}\left(  U\right)  }\right\}  .
\]
Here the H\"{o}lder spaces $C^{k,\alpha}$ are those defined by means of
derivatives with respect to the $X_{i}$'s and the distance induced by the
vector fields (more precisely, the parabolic version of these spaces, with the
time derivative weighting as a second order derivative, see
\S \ref{sec Hormander} and \S \ref{sec evolution} for precise definitions).

Note that the previous estimate assumes \textit{a priori }that $u\in
C_{loc}^{k+2,\alpha}\left(  U\right)  $. A more subtle problem is that of
proving a regularity result of the kind: if $u\in C^{2,\alpha}\left(
U\right)  $ solves $Hu=f$ and $a_{ij},b_{i},c,f\in C^{k,\alpha}\left(
U\right)  $ then actually $u\in C_{loc}^{k+2,\alpha}\left(  U\right)  $ (and
therefore the above a priori estimate holds). In \cite{BB2007} this regularity
result is actually proved, applying the classical strategy of regularizing the
coefficients and data of the equation, solving the regularized Dirichlet
problem and exploiting the a priori estimate to build a sequence of smooth
functions converging in $C_{loc}^{k+2,\alpha}\left(  U\right)  $ to the
solution of $Hu=f$. However, using this approach in \cite{BB2007} the
boundedness of the approximating sequence is proved only for $k$ even, hence
the regularity result has been proved so far only for $k$ even. In the present
paper, exploiting the a priori estimates proved in \cite{BB2007}, we find a
different way of proving a regularity result which holds for all $k$ (see
Theorem \ref{thm main linear}), based on the Banach-Caccioppoli fixed point
theorem. The above result and a standard bootstrap argument enable us to prove
a Schauder regularity result for quasilinear equations of the kind%
\[
Qu\equiv\sum_{i,j=1}^{q}a_{ij}\left(  x,u,Xu\right)  X_{i}X_{j}u+b\left(
x,u,Xu\right)  =0,
\]
with $\left\{  a_{ij}\right\}  $ uniformly positive, concluding that, in
particular, any $C^{2,\alpha}\left(  \Omega\right)  $ solution to $Qu=0$ is
smooth as soon as $a_{ij},b$ are smooth (see Theorem
\ref{Thm main quasilinear}). Finally, in view of the results in \cite{BB2007},
both the linear and the quasilinear regularity results described above can be
easily extended to evolution operators $\partial_{t}-L$, $\partial_{t}-Q$ (see
Theorems \ref{Thm linear evolution}, \ref{Thm evolution quasilinear}).
Actually, we have written our proofs in the stationary case just to simplify notation.

We mention that in the paper \cite{Xu} this regularity result is also stated,
but the Author makes an extra assumption on the structure of the $X_{i}$'s
(which does not cover general H\"{o}rmander's vector fields) and he actually
proves only local a priori estimates, not a regularity result.

\section{Preliminaries and known results}

\subsection{H\"{o}rmander's vector fields, control distance and H\"{o}lder
spaces\label{sec Hormander}}

Let $X_{1},\ldots,X_{q}$ be a system of real smooth vector fields,%
\[
X_{i}=\sum_{j=1}^{n}b_{ij}\left(  x\right)  \partial_{x_{j}},\text{
\ }i=1,2,...,q
\]
$\left(  q<n\right)  $ defined in some bounded, open and connected subset
$\Omega_{0}$ of $\mathbb{R}^{n}$. For any multiindex%
\[
I=(i_{1},i_{2},...,i_{k}),\text{ \ \ }1\leq i_{j}\leq q
\]
we set:%
\[
X_{[I]}=\left[  X_{i_{1}},\left[  X_{i_{2}},\ldots\left[  X_{i_{k-1}}%
,X_{i_{k}}\right]  \ldots\right]  \right]  ,
\]
where $\left[  X,Y\right]  =XY-YX$\ for any couple of vector fields $X,Y$. We
will say that $X_{[I]}$ is a \emph{commutator of length }$|I|=k.$ As usual,
$X_{i}$ can be seen either as a differential operator or as a vector field. We
will write $X_{i}f$ to denote the differential operator $X_{i}$ acting on a
function $f$, and $\left(  X_{i}\right)  _{x}$ to denote the vector field
$X_{i}$ evaluated at the point $x\in\Omega_{0}$. We shall say that
$X_{1},\ldots,X_{q}$ satisfy \emph{H\"{o}rmander's condition of step }$s$ in
$\Omega_{0}$ if these vector fields, together with their commutators of length
$\leq s$, span the tangent space at every point $x\in\Omega_{0}$.

One can now define the control distance induced by these vector fields, as in
\cite{NSW}:

\begin{definition}
\label{c-c-distance}For any $\delta>0$, let $C\left(  \delta\right)  $ be the
class of absolutely continuous mappings $\varphi:[0,1]\rightarrow\Omega_{0}$
which satisfy%
\[
\varphi^{\prime}(t)=\overset{q}{\underset{i=1}{\sum}}\lambda_{i}(t)\left(
X_{i}\right)  _{\varphi\left(  t\right)  }\text{ a.e. }t\in\left(  0,1\right)
\]
with $\left\vert \lambda_{j}(t)\right\vert \leq\delta$ for $j=1,...,q$. We
define%
\[
d_{X}(x,y)=\inf\left\{  \delta:\exists\varphi\in C\left(  \delta\right)
\text{ with }\varphi\left(  0\right)  =x,\varphi\left(  1\right)  =y\right\}
.
\]

\end{definition}

Note that the finiteness of $d_{X}\left(  x,y\right)  $ for any two points
$x,y\in\Omega_{0}$ is not a trivial fact, but depends on a connectivity result
(\textquotedblleft Chow's theorem\textquotedblright); moreover, it can be
proved that $d_{X}$ is a distance. It is also well-known that this distance is
topologically equivalent to the Euclidean one (see \cite{NSW} for all these facts).

Now, let $\Omega\subset\Omega_{0}$ be another fixed domain. For any\ $x\in
\Omega$, we set%
\[
B_{r}\left(  x\right)  =\left\{  y\in\Omega_{0}:d_{X}\left(  x,y\right)
<r\right\}  .
\]

Let us define several types of H\"{o}lder spaces that we will need in the following:

\begin{definition}
For any $\alpha\in\left(  0,1\right)  ,u:\Omega\rightarrow\mathbb{R}$, let:%
\begin{align*}
\left\vert u\right\vert _{C_{X}^{\alpha}\left(  \Omega\right)  }  &
=\sup\left\{  \frac{\left\vert u\left(  x\right)  -u\left(  y\right)
\right\vert }{d_{X}\left(  x,y\right)  ^{\alpha}}:x,y\in\Omega,x\neq
y\right\}  ,\\
\left\Vert u\right\Vert _{C_{X}^{\alpha}\left(  \Omega\right)  }  &
=\left\vert u\right\vert _{C^{\alpha}\left(  \Omega\right)  }+\left\Vert
u\right\Vert _{L^{\infty}\left(  \Omega\right)  },\\
C_{X}^{\alpha}\left(  \Omega\right)   &  =\left\{  u:\Omega\rightarrow
\mathbb{R}:\left\Vert u\right\Vert _{C^{\alpha}\left(  \Omega\right)  }%
<\infty\right\}  .
\end{align*}
Also, for any positive integer $k$, let%
\[
C_{X}^{k,\alpha}\left(  \Omega\right)  =\left\{  u:\Omega\rightarrow
\mathbb{R}:\left\Vert u\right\Vert _{C^{k,\alpha}\left(  \Omega\right)
}<\infty\right\}  ,\text{ \ }%
\]
with%
\[
\left\Vert u\right\Vert _{C_{X}^{k,\alpha}\left(  \Omega\right)  }%
=\underset{l=1}{\overset{k}{\sum}}\underset{j_{i}=1}{\overset{q}{\sum}%
}\left\Vert X_{j_{1}}\ldots X_{j_{l}}u\right\Vert _{C^{\alpha}\left(
\Omega\right)  }+\left\Vert u\right\Vert _{C^{\alpha}\left(  \Omega\right)
}.
\]

We will set $C_{X,0}^{\alpha}\left(  \Omega\right)  $ and $C_{X,0}^{k,\alpha
}\left(  \Omega\right)  $ for the subspaces of $C_{X}^{\alpha}\left(
\Omega\right)  $ and $C_{X}^{k,\alpha}\left(  \Omega\right)  $ of functions
which are compactly supported in $\Omega$, and $C_{X,loc}^{k,\alpha}\left(
\Omega\right)  $ for the space of functions belonging to $C_{X}^{k,\alpha
}\left(  \Omega^{\prime}\right)  $ for every $\Omega^{\prime}\Subset\Omega$.

Finally, we will write $C_{X,\ast}^{k,\alpha}\left(  \Omega\right)  $ to
denote the subspace of $C_{X}^{k,\alpha}\left(  \Omega\right)  $ consisting of
functions $u$ such that both $u$ and all the derivatives $X_{i_{1}}X_{i_{2}%
}...X_{i_{l}}u$ $\left(  l\leq k\right)  $ vanish on $\partial\Omega.$
\end{definition}

\begin{proposition}
\label{prop extension}The spaces $C_{X,\ast}^{k,\alpha}\left(  \Omega\right)
$ are complete. Moreover, if $u\in C_{X,\ast}^{k,\alpha}\left(  B_{r}\left(
x_{0}\right)  \right)  $ and $R>r$ for some $B_{R}\left(  x_{0}\right)
\subset\Omega,$ then
\[
\overline{u}\left(  x\right)  =\left\{
\begin{array}
[c]{l}%
u\left(  x\right)  \text{ in }B_{r}\left(  x_{0}\right) \\
0\text{ in }B_{R}\left(  x_{0}\right)  \setminus B_{r}\left(  x_{0}\right)
\end{array}
\right.
\]
belongs to $C_{X,0}^{k,\alpha}\left(  B_{R}\left(  x_{0}\right)  \right)  .$
\end{proposition}

\begin{proof}
We leave to the reader to check the completeness of these spaces. Let us prove
the second assertion for $k=0$, since the same argument applies to the
derivatives. It is enough to check that%
\[
\left\vert \overline{u}\left(  x\right)  -\overline{u}\left(  y\right)
\right\vert \leq cd_{X}\left(  x,y\right)  ^{\alpha}\text{ for any }x\in
B_{r}\left(  x_{0}\right)  ,y\in B_{R}\left(  x_{0}\right)  \setminus
B_{r}\left(  x_{0}\right)  ,
\]
the other cases being obvious. By definition of $d_{X}$, for any fixed
$\varepsilon>0$, we can pick a curve $\gamma:\left[  0,1\right]
\rightarrow\Omega$ such that%
\begin{align*}
\gamma\left(  0\right)   &  =x,\gamma\left(  1\right)  =y\\
\gamma^{\prime}\left(  t\right)   &  =\sum_{i=1}^{q}\lambda_{i}\left(
t\right)  \left(  X_{i}\right)  _{\gamma\left(  t\right)  }\text{ with
}\left\vert \lambda_{i}\left(  t\right)  \right\vert \leq d_{X}\left(
x,y\right)  +\varepsilon.
\end{align*}
Since $d_{X}\left(  x_{0},\gamma\left(  0\right)  \right)  <r$ and
$d_{X}\left(  x_{0},\gamma\left(  1\right)  \right)  >r,$ there exists a point
$z=\gamma\left(  \overline{t}\right)  $ such that $d\left(  x_{0},z\right)
=r$ and $u\left(  z\right)  =0$. Then%
\begin{align*}
\left\vert \overline{u}\left(  x\right)  -\overline{u}\left(  y\right)
\right\vert  &  =\left\vert u\left(  x\right)  \right\vert =\left\vert
u\left(  x\right)  -u\left(  z\right)  \right\vert \\
&  \leq\left\vert u\right\vert _{C^{\alpha}}d_{X}\left(  x,z\right)  ^{\alpha
}\leq\left\vert u\right\vert _{C^{\alpha}}\left(  d_{X}\left(  x,y\right)
+\varepsilon\right)  ^{\alpha}.
\end{align*}
Since this is true for every $\varepsilon>0,$ we are done.
\end{proof}

The following easy properties of our function spaces will be useful:

\begin{proposition}
\label{basic properties for holder norm}(See \cite[Proposition 4.2]{BB2007},
also \cite[Prop.3.27]{BZ2}) Let $B_{R}\left(  \overline{x}\right)
\subset\Omega$, then

(i) For any $f\in C_{X,0}^{1}\left(  B_{R}\left(  \overline{x}\right)
\right)  $, one has
\begin{equation}
\left\vert f(x)-f(y)\right\vert \leq d_{X}\left(  x,y\right)  \sum_{i=1}%
^{q}\underset{B_{R}\left(  \overline{x}\right)  }{\sup}\left\vert
X_{i}f\right\vert \label{4.6}%
\end{equation}
for any $x,y\in B_{R}\left(  \overline{x}\right)  $.

(ii) For any couple of functions $f,g\in C_{X}^{\alpha}\left(  B_{R}\left(
\overline{x}\right)  \right)  $, one has
\[
\left\vert fg\right\vert _{C_{X}^{\alpha}\left(  B_{R}\left(  \overline
{x}\right)  \right)  }\leq\left\vert f\right\vert _{C_{X}^{\alpha}\left(
B_{R}\left(  \overline{x}\right)  \right)  }\left\Vert g\right\Vert
_{L^{\infty}\left(  B_{R}\left(  \overline{x}\right)  \right)  }+\left\vert
g\right\vert _{C_{X}^{\alpha}\left(  B_{R}\left(  \overline{x}\right)
\right)  }\left\Vert f\right\Vert _{L^{\infty}\left(  B_{R}\left(
\overline{x}\right)  \right)  }\text{ }%
\]
and%
\begin{equation}
\left\Vert fg\right\Vert _{C_{X}^{\alpha}\left(  B_{R}\left(  \overline
{x}\right)  \right)  }\leq2\left\Vert f\right\Vert _{C_{X}^{\alpha}\left(
B_{R}\left(  \overline{x}\right)  \right)  }\left\Vert g\right\Vert
_{C_{X}^{\alpha}\left(  B_{R}\left(  \overline{x}\right)  \right)  }.
\label{1.2}%
\end{equation}

\end{proposition}

\subsection{Lifted vector fields and integral operators}

Throughout the paper we will make use of some results and techniques
originally introduced by Rothschild-Stein \cite{rs} and then adapted to
nonvariational operators in \cite{BB2000}, \cite{BB2007}, \cite{BZ2}. However,
in order to understand the proofs in the present paper, it is not necessary
for the reader to know in detail all the background which is implicitly
involved here. Therefore, to reduce the length of this paper we will content
ourselves of pointing out the facts which will be explicitly used, giving to
the interested reader all the relevant references.

First of all, much of the proof of our main results lives in the space of
\textquotedblleft lifted variables\textquotedblright, as in \cite{rs}. This
basically means what follows. For every point $\overline{x}\in\Omega$ there
exists a neighborhood $B_{R}\left(  \overline{x}\right)  \subset\Omega$ and,
in terms of new variables, $h_{n+1},\ldots,h_{N}$, there exist smooth
functions $\lambda_{il}(x,h)$ ($1\leq i\leq q,$ $n+1\leq l\leq N$) defined in
a neighborhood $\widetilde{U}$ of $\overline{\xi}=\left(  \overline
{x},0\right)  \in\mathbb{R}^{N}$ such that the vector fields $\widetilde
{X}_{i}$ given by%
\[
\widetilde{X}_{i}=X_{i}+\underset{l=n+1}{\overset{N}{\sum}}\lambda
_{il}(x,h)\frac{\partial}{\partial h_{l}},\ \ i=1,\ldots,q
\]
still satisfy H\"{o}rmander's condition of step $s$ in $\widetilde{U}$ and
possesses further properties, some of which we are going to recall.

Let us first fix some notation. We will denote by $d_{\widetilde{X}}$ the
control distance induced by the vector fields $\widetilde{X}_{i}$ in
$\widetilde{U}$, by $\widetilde{B}_{R}\left(  \overline{\xi}\right)  $ the
corresponding balls, and we will denote by $C_{\widetilde{X}}^{\alpha}\left(
\widetilde{B}_{R}\left(  \overline{\xi}\right)  \right)  $, $C_{\widetilde{X}%
}^{k,\alpha}\left(  \widetilde{B}_{R}\left(  \overline{\xi}\right)  \right)
$, $C_{\widetilde{X},0}^{\alpha}\left(  \widetilde{B}_{R}\left(  \overline
{\xi}\right)  \right)  $ and $C_{\widetilde{X},0}^{k,\alpha}\left(
\widetilde{B}_{R}\left(  \overline{\xi}\right)  \right)  $ the function spaces
over $\widetilde{B}_{R}\left(  \overline{\xi}\right)  $ defined by the
$\widetilde{X}_{i}$'s as in \S \ref{sec Hormander}.

The following relation between the spaces $C_{X}^{\alpha}\left(  B_{R}\left(
\overline{x}\right)  \right)  $ and $C_{\widetilde{X}}^{\alpha}\left(
\widetilde{B}_{R}\left(  \overline{\xi}\right)  \right)  $ is crucial for us:

\begin{proposition}
\label{Proposition de-lifting}(See \cite[Prop. 8.3]{BB2007}, \cite[Prop.
3.28]{BZ2}) Let $\widetilde{B}_{r}\left(  \overline{\xi}\right)  $ be a lifted
ball, with $\overline{\xi}=\left(  \overline{x},0\right)  $. If $f$ is a
function defined in $B_{2r}\left(  \overline{x}\right)  $ and $\widetilde
{f}\left(  x,h\right)  =f\left(  x\right)  $ is regarded as a function defined
on $\widetilde{B}_{r}\left(  \overline{\xi}\right)  $, then the following
inequalities hold true%
\[
\left\vert \widetilde{f}\right\vert _{C_{\widetilde{X}}^{\alpha}\left(
\widetilde{B}_{r}\left(  \overline{\xi}\right)  \right)  }\leqslant\left\vert
f\right\vert _{C_{X}^{\alpha}\left(  B_{r}\left(  \overline{x}\right)
\right)  }\leqslant c\left\vert \widetilde{f}\right\vert _{C_{\widetilde{X}%
}^{\alpha}\left(  \widetilde{B}_{2r}\left(  \overline{\xi}\right)  \right)
}.
\]
Moreover,%
\[
\left\vert \widetilde{X}_{i_{1}}\widetilde{X}_{i_{2}}...\widetilde{X}_{i_{k}%
}\widetilde{f}\right\vert _{C_{\widetilde{X}}^{\alpha}\left(  \widetilde
{B}_{r}\left(  \overline{\xi}\right)  \right)  }\leqslant\left\vert X_{i_{1}%
}...X_{i_{k}}f\right\vert _{C_{X}^{\alpha}\left(  B_{r}\left(  \overline
{x}\right)  \right)  }\leqslant c\left\vert \widetilde{X}_{i_{1}}%
...\widetilde{X}_{i_{k}}\widetilde{f}\right\vert _{C_{\widetilde{X}}^{\alpha
}\left(  \widetilde{B}_{2r}\left(  \overline{\xi}\right)  \right)  }\text{ }%
\]
for $i_{j}=1,2,...,q.$
\end{proposition}

The main tool to prove a priori estimates (in \cite{rs}, \cite{BB2000},
\cite{BB2007}, \cite{BZ2}) is the combination of some abstract theory of
singular integrals with some representation formulas for the second order
derivatives $\widetilde{X}_{i}\widetilde{X}_{j}u$ of any test function by
means of suitable integral operators. The reason why this is performed in the
spaces of lifted variables is that this allows to make use of singular
integral operators with better properties. The key notion here is that of
\emph{frozen operator of type zero over a ball} $\widetilde{B}_{R}\left(
\xi_{0}\right)  $, first introduced in \cite{BB2000} adapting the notion of
operator of type zero given in \cite{rs}. We will not recall the definition of
this concept (see \cite[Definition 6.3]{BB2007}) because it involves several
other notions that we will not use explicitly. It is enough to say that a
frozen operator of type zero over $\widetilde{B}_{R}\left(  \overline{\xi
}\right)  $ is an integral operator $T\left(  \xi_{0}\right)  $ (depending on
some point $\xi_{0}\in\widetilde{B}_{R}\left(  \overline{\xi}\right)  $ like a
parameter), and that the following two results hold:

\begin{theorem}
\label{Thm type zero continuity}(see \cite[Thm.6.6]{BB2007}, see also
\cite[Thm.5.1]{BZ2}). There exists $C_{R}>0$ depending on $R,\Omega$ and the
vector fields $X_{i}$ (but not on $\xi_{0}$) such that for every $r\leq R,f\in
C_{X,0}^{\alpha}\left(  \widetilde{B}_{r}\left(  \overline{\xi}\right)
\right)  ,$%
\[
\left\Vert T\left(  \xi_{0}\right)  f\right\Vert _{C_{X}^{\alpha}\left(
\widetilde{B}_{r}\left(  \overline{\xi}\right)  \right)  }\leq C_{R}\left\Vert
f\right\Vert _{C_{X}^{\alpha}\left(  \widetilde{B}_{r}\left(  \overline{\xi
}\right)  \right)  }.
\]

\end{theorem}

\begin{theorem}
\label{thm type 0 exchange}For any $k=1,2,...,q$ there exist $q+1$ frozen
operators of type zero over $\widetilde{B}_{R}\left(  \overline{\xi}\right)
$, $T_{h}^{k}\left(  \xi_{0}\right)  $ for $k=0,1,2,...,q,$ such that for any
$f\in C_{X}^{1}\left(  \widetilde{B}_{r}\left(  \overline{\xi}\right)
\right)  $ one has:%
\[
\widetilde{X}_{k}T\left(  \xi_{0}\right)  f=\sum_{h=1}^{q}T_{h}^{k}\left(
\xi_{0}\right)  \widetilde{X}_{h}f+T_{k}^{0}\left(  \xi_{0}\right)  f
\]

\end{theorem}

\begin{proof}
[Proof of Theorem \ref{thm type 0 exchange}]In \cite[Prop.6.9]{BB2007} an
analogous formula is stated for $T\left(  \xi_{0}\right)  ,T_{h}^{k}\left(
\xi_{0}\right)  $ frozen operators \emph{of type one}. Exploiting the fact
that any frozen operator of type zero involved in our argument is actually of
the kind $\widetilde{X}_{i}S\left(  \xi_{0}\right)  $ with $S\left(  \xi
_{0}\right)  $ frozen operator of type one, and \emph{viceversa }for every
frozen operator of type one $S\left(  \xi_{0}\right)  $ the derivative
$\widetilde{X}_{i}S\left(  \xi_{0}\right)  $ is a frozen operator of type
zero, we can write (denoting frozen operators of type zero or one with the
letters $T,S,$ respectively):%
\begin{align*}
\widetilde{X}_{k}T\left(  \xi_{0}\right)  f &  =\widetilde{X}_{k}\widetilde
{X}_{i}S\left(  \xi_{0}\right)  f=\widetilde{X}_{k}\left(  \sum_{h=1}^{q}%
S_{h}^{i}\left(  \xi_{0}\right)  \widetilde{X}_{h}f+S_{k}^{0}\left(  \xi
_{0}\right)  f\right)  \\
&  =\sum_{h=1}^{q}T_{h}^{k}\left(  \xi_{0}\right)  \widetilde{X}_{h}%
f+T_{k}^{0}\left(  \xi_{0}\right)  f,
\end{align*}
which gives the assertion.
\end{proof}

\section{The linear regularity theory\label{sec linear}}

Let us consider the linear operator%
\[
L=\sum_{i,j=1}^{q}a_{ij}\left(  x\right)  X_{i}X_{j}+\sum_{i=1}^{q}%
b_{i}\left(  x\right)  X_{i}+c\left(  x\right)
\]
where:

$X_{1},X_{2},...,X_{q}$ are a system of H\"{o}rmander's vector fields in a
neighborhood $\Omega_{0}$ of some bounded domain $\Omega\subset\mathbb{R}^{n}%
$, as described at the beginning of \S \ \ref{sec Hormander};

$a_{ij},b_{i},c\in C_{X}^{\alpha}\left(  \Omega\right)  $ for some $\alpha
\in\left(  0,1\right)  $, $a_{ij}=a_{ji}$ satisfying for some constant
$\Lambda>0$ the condition%
\begin{equation}
\Lambda|\xi|^{2}\leq\sum_{i,j=1}^{q}a_{ij}\left(  x\right)  \xi_{i}\xi_{j}%
\leq\Lambda^{-1}|\xi|^{2}\text{ }\forall\xi\in{\mathbb{R}}^{q},x\in\Omega.
\label{ellipticity}%
\end{equation}

The aim of this section is to prove the following result:

\begin{theorem}
\label{thm main linear}Under the above assumptions, let $u\in C_{X}^{2,\alpha
}\left(  \Omega\right)  $ satisfy the equation%
\[
Lu=f\text{ in }\Omega
\]
and assume that for some integer $k=1,2,3,...$we have:%
\[
a_{ij},b_{i},c,f,\in C_{X}^{k,\alpha}\left(  \Omega\right)  .
\]
Then%
\[
u\in C_{X,loc}^{k+2,\alpha}\left(  \Omega\right)  .
\]
In particular, if
\[
a_{ij},b_{i},c,f,\in C^{\infty}\left(  \Omega\right)  \text{ }%
\]
then
\[
u\in C^{\infty}\left(  \Omega\right)  .
\]

\end{theorem}

In virtue of the results in \cite{BB2007}, the regularity $u\in C_{X,loc}%
^{k+2,\alpha}\left(  \Omega\right)  $ also implies the validity of local a
priori estimates
\[
\left\Vert u\right\Vert _{C_{X}^{k+2,\alpha}\left(  \Omega^{\prime}\right)
}\leqslant c\left\{  \left\Vert Lu\right\Vert _{C_{X}^{k,\alpha}\left(
\Omega\right)  }+\left\Vert u\right\Vert _{L^{\infty}\left(  \Omega\right)
}\right\}
\]
for any $\Omega^{\prime}\Subset\Omega$, with constant $c$ independent of $u$.

\begin{remark}
\label{remark principal part}Clearly, it is enough to prove the theorem for
$b_{i}=c=0$, because assuming this we can proceed as follows: let $u\in
C_{X}^{2,\alpha}\left(  \Omega\right)  $ satisfy the equation%
\[
Lu=f\text{ in }\Omega
\]
and assume that%
\[
a_{ij},b_{i},c,f,\in C_{X}^{1,\alpha}\left(  \Omega\right)  .
\]
Then%
\[
\sum_{i,j=1}^{q}a_{ij}\left(  x\right)  X_{i}X_{j}u=f-\sum_{i=1}^{q}%
b_{i}\left(  x\right)  X_{i}u-c\left(  x\right)  u\equiv\widetilde{f}\in
C_{X}^{1,\alpha}\left(  \Omega\right)
\]
and by the result that we suppose already proved for the principal part
operator we conclude $u\in C_{X,loc}^{3,\alpha}\left(  \Omega\right)  $.
Iterating this argument gives the general result for any $k$. Hence, we need
to prove Theorem \ref{thm main linear} only for $b_{i}=c=0.$
\end{remark}

Now, fix $x_{0}\in\Omega$ and a small ball $B_{R}\left(  x_{0}\right)
\subset\Omega$ where the lifting procedure is applicable. Let $\xi_{0}=\left(
x_{0},0\right)  ,\xi=\left(  x,h\right)  $, and define, for $\xi\in
\widetilde{B}_{R}\left(  \xi_{0}\right)  ,$%
\begin{align*}
\widetilde{a}_{ij}\left(  \xi\right)   &  =a_{ij}\left(  x\right) \\
\widetilde{L}u\left(  \xi\right)   &  =\sum_{i,j=1}^{q}\widetilde{a}%
_{ij}\left(  \xi\right)  \widetilde{X}_{i}\widetilde{X}_{j}u\left(
\xi\right)  .
\end{align*}
Next, let us freeze the coefficients $\widetilde{a}_{ij}$ at $\xi_{0},$ and
let%
\[
\widetilde{L}_{0}u\left(  \xi\right)  =\sum_{i,j=1}^{q}\widetilde{a}%
_{ij}\left(  \xi_{0}\right)  \widetilde{X}_{i}\widetilde{X}_{j}u\left(
\xi\right)
\]

For this frozen lifted operator the following representation formula holds true:

\begin{theorem}
[Representation of $\widetilde{X}_{m}\widetilde{X}_{l}u$ by frozen
operators]\label{Thm rep form C^alfa}(\cite[p.211]{BB2007}) Given $a\in
C_{0}^{\infty}\left(  \widetilde{B}_{R}\left(  \xi_{0}\right)  \right)  $,
there exist frozen operators $T_{lm}\left(  \xi_{0}\right)  $ over the ball
$\widetilde{B}_{R}\left(  \xi_{0}\right)  $ ($m,l=1,2,...,q$), such that for
any $u\in C_{X,0}^{2,\alpha}\left(  \widetilde{B}_{R}\left(  \xi_{0}\right)
\right)  $%
\begin{equation}
\widetilde{X}_{m}\widetilde{X}_{l}\left(  au\right)  =T_{lm}\left(  \xi
_{0}\right)  \widetilde{\mathcal{L}}_{0}u+\sum_{i,j=1}^{q}\widetilde{a}%
_{ij}\left(  \xi_{0}\right)  \left\{  \sum_{k=1}^{q}T_{lm,k}^{ij}\left(
\xi_{0}\right)  \widetilde{X}_{k}u+T_{lm}^{ij}\left(  \xi_{0}\right)
u\right\}  .\nonumber
\end{equation}
Also,
\begin{align*}
&  \widetilde{X}_{m}\widetilde{X}_{l}\left(  au\right)  =T_{lm}\left(  \xi
_{0}\right)  \widetilde{\mathcal{L}}u+T_{lm}\left(  \xi_{0}\right)  \left(
\sum_{i,j=1}^{q}\left[  \widetilde{a}_{ij}(\xi_{0})-\widetilde{a}_{ij}%
(\cdot)\right]  \,\widetilde{X}_{i}\widetilde{X}_{j}u\right)  +\\
&  +\sum_{i,j=1}^{q}\widetilde{a}_{ij}\left(  \xi_{0}\right)  \left\{
\sum_{k=1}^{q}T_{lm,k}^{ij}\left(  \xi_{0}\right)  \widetilde{X}_{k}%
u+T_{lm}^{ij}\left(  \xi_{0}\right)  u\right\}  .
\end{align*}

\end{theorem}

In order to make more readable the previous formulas, let us define:%
\begin{align*}
T_{lm,k}^{A}\left(  \xi_{0}\right)   &  =\sum_{i,j=1}^{q}\widetilde{a}%
_{ij}\left(  \xi_{0}\right)  T_{lm,k}^{ij}\left(  \xi_{0}\right) \\
T_{lm}^{A}\left(  \xi_{0}\right)   &  =\sum_{i,j=1}^{q}\widetilde{a}%
_{ij}\left(  \xi_{0}\right)  T_{lm}^{ij}\left(  \xi_{0}\right)
\end{align*}
hence%
\begin{align}
\widetilde{X}_{m}\widetilde{X}_{l}\left(  au\right)   &  =T_{lm}\left(
\xi_{0}\right)  \widetilde{\mathcal{L}}u+T_{lm}\left(  \xi_{0}\right)  \left(
\sum_{i,j=1}^{q}\left[  \widetilde{a}_{ij}(\xi_{0})-\widetilde{a}_{ij}%
(\cdot)\right]  \,\widetilde{X}_{i}\widetilde{X}_{j}u\right)  \label{id basic}%
\\
&  +\sum_{k=1}^{q}T_{lm,k}^{A}\left(  \xi_{0}\right)  \widetilde{X}%
_{k}u+T_{lm}^{A}\left(  \xi_{0}\right)  u.\nonumber
\end{align}

\begin{remark}
We note that by Theorem \ref{Thm type zero continuity}, the operators
$T_{lm,k}^{A}\left(  \xi_{0}\right)  ,T_{lm}^{A}\left(  \xi_{0}\right)  $
satisfy the estimate:
\begin{equation}
\left\Vert T_{...}^{A}\left(  \xi_{0}\right)  f\right\Vert _{C^{\alpha}\left(
\widetilde{B}_{r}\left(  \xi_{0}\right)  \right)  }\leq C_{R,\Lambda
}\left\Vert f\right\Vert _{C^{\alpha}\left(  \widetilde{B}_{r}\left(  \xi
_{0}\right)  \right)  } \label{continuity}%
\end{equation}
for any $f\in C_{X,0}^{\alpha}\left(  \widetilde{B}_{r}\left(  \xi_{0}\right)
\right)  $, where now the constant $C_{R,\Lambda}$ also depends on the number
$\Lambda$ in (\ref{ellipticity}).
\end{remark}

We are going to see (\ref{id basic}) as an identity involving a suitable
integral operator, to which apply the Banach-Caccioppoli fixed point theorem.
To this aim, for a fixed $v\in C_{X,0}^{2,\alpha}\left(  \widetilde{B}\left(
\xi_{0},R\right)  \right)  $ let%
\begin{equation}
G_{l,m}=T_{lm}\left(  \xi_{0}\right)  \widetilde{\mathcal{L}}v+\sum_{k=1}%
^{q}T_{lm,k}^{A}\left(  \xi_{0}\right)  \widetilde{X}_{k}v+T_{lm}^{A}\left(
\xi_{0}\right)  v \label{G_l,m}%
\end{equation}
hence by (\ref{continuity})%
\[
G_{l,m}\in C_{X}^{\alpha}\left(  \widetilde{B}\left(  \xi_{0},R\right)
\right)
\]
and%
\[
\widetilde{X}_{m}\widetilde{X}_{l}\left(  av\right)  =G_{l,m}+T_{lm}\left(
\xi_{0}\right)  \left(  \sum_{i,j=1}^{q}\left[  \widetilde{a}_{ij}(\xi
_{0})-\widetilde{a}_{ij}(\cdot)\right]  \,\widetilde{X}_{i}\widetilde{X}%
_{j}v\right)  .
\]
Now, for a number $r<R$ to be fixed later, pick another $\beta\in
C_{0}^{\infty}\left(  \widetilde{B}_{r}\left(  \xi_{0}\right)  \right)  $ such
that $\beta=1$ in $\widetilde{B}_{r/2}\left(  \xi_{0}\right)  $, and write%
\begin{equation}
\beta\widetilde{X}_{m}\widetilde{X}_{l}\left(  av\right)  =\beta G_{l,m}+\beta
T_{lm}\left(  \xi_{0}\right)  \left(  \sum_{i,j=1}^{q}\left[  \widetilde
{a}_{ij}(\xi_{0})-\widetilde{a}_{ij}(\cdot)\right]  \,\widetilde{X}%
_{i}\widetilde{X}_{j}v\right)  . \label{rep formula}%
\end{equation}
Now, for any $F=\left(  F_{ij}\right)  _{i,j=1}^{q}\in\left(  C_{X,\ast
}^{\alpha}\left(  \widetilde{B}_{r}\left(  \xi_{0}\right)  \right)  \right)
^{q\times q},$ let us define the operator
\[
\mathcal{T}\left(  F\right)  =\beta G_{l,m}+\beta T_{lm}\left(  \xi
_{0}\right)  \left(  \sum_{i,j=1}^{q}\left[  \widetilde{a}_{ij}(\xi
_{0})-\widetilde{a}_{ij}(\cdot)\right]  \,F_{ij}\right)  .
\]

\begin{theorem}
\label{thm contraction 0}For $r>0$ small enough, the operator $\mathcal{T}$ is
a contraction of $\left(  C_{X,\ast}^{\alpha}\left(  B_{r}\left(  \xi
_{0}\right)  \right)  \right)  ^{q\times q}$ in itself.
\end{theorem}

\begin{proof}
Since $F_{ij}\in C_{\ast}^{\alpha}\left(  \widetilde{B}_{r}\left(  \xi
_{0}\right)  \right)  $, $F_{ij}$ can be extended to zero in $\widetilde
{B}_{R}\left(  \xi_{0}\right)  $, hence (see Proposition \ref{prop extension})%
\[
\sum_{i,j=1}^{q}\left[  \widetilde{a}_{ij}(\xi_{0})-\widetilde{a}_{ij}%
(\cdot)\right]  \,F_{ij}\in C_{X,0}^{\alpha}\left(  \widetilde{B}_{R}\left(
\xi_{0}\right)  \right)  ,
\]
and by Theorem \ref{Thm type zero continuity},
\[
T_{lm}\left(  \xi_{0}\right)  \left(  \sum_{i,j=1}^{q}\left[  \widetilde
{a}_{ij}(\xi_{0})-\widetilde{a}_{ij}(\cdot)\right]  \,F_{ij}\right)  \in
C_{X}^{\alpha}\left(  \widetilde{B}_{R}\left(  \xi_{0}\right)  \right)
\]
and%
\[
\beta T_{lm}\left(  \xi_{0}\right)  \left(  \sum_{i,j=1}^{q}\left[
\widetilde{a}_{ij}(\xi_{0})-\widetilde{a}_{ij}(\cdot)\right]  \,F_{ij}\right)
\in C_{X,\ast}^{\alpha}\left(  \widetilde{B}_{r}\left(  \xi_{0}\right)
\right)  .
\]

Since also $G_{l,m}\in C_{X}^{\alpha}\left(  \widetilde{B}_{R}\left(  \xi
_{0}\right)  \right)  $ and $\beta G_{l,m}\in C_{X,\ast}^{\alpha}\left(
\widetilde{B}_{r}\left(  \xi_{0}\right)  \right)  $, we conclude that
$\mathcal{T}$ maps $\left(  C_{X,\ast}^{\alpha}\left(  B_{r}\left(  \xi
_{0}\right)  \right)  \right)  ^{q\times q}$ in itself. In order to show that
$\mathcal{T}$ is a contraction, let $F^{\left(  1\right)  },F^{\left(
2\right)  }\in\left(  C_{X,\ast}^{\alpha}\left(  B_{r}\left(  \xi_{0}\right)
\right)  \right)  ^{q\times q}$. We have, by Theorem
\ref{Thm type zero continuity}\ and (\ref{1.2}):%
\begin{align*}
&  \left\Vert \mathcal{T}F^{\left(  1\right)  }-\mathcal{T}F^{\left(
2\right)  }\right\Vert _{\left(  C_{X}^{\alpha}\left(  \widetilde{B}%
_{r}\left(  \xi_{0}\right)  \right)  \right)  ^{q\times q}}\\
&  \leq\sum_{l,m=1}^{q}\left\Vert \beta T_{lm}\left(  \xi_{0}\right)
\sum_{i,j=1}^{q}\left[  \widetilde{a}_{ij}(\xi_{0})-\widetilde{a}_{ij}%
(\cdot)\right]  \,\left(  F_{ij}^{\left(  1\right)  }-F_{ij}^{\left(
2\right)  }\right)  \right\Vert _{C_{X}^{\alpha}\left(  \widetilde{B}%
_{r}\left(  \xi_{0}\right)  \right)  }\\
&  \leq\sum_{l,m=1}^{q}\sum_{i,j=1}^{q}c\left\Vert \left[  \widetilde{a}%
_{ij}(\xi_{0})-\widetilde{a}_{ij}(\cdot)\right]  \left(  F_{ij}^{\left(
1\right)  }-F_{ij}^{\left(  2\right)  }\right)  \right\Vert _{C_{X}^{\alpha
}\left(  \widetilde{B}_{r}\left(  \xi_{0}\right)  \right)  }\\
&  \leq c\omega\left(  r\right)  \left\Vert F^{\left(  1\right)  }-F^{\left(
2\right)  }\right\Vert _{\left(  C_{X}^{\alpha}\left(  \widetilde{B}%
_{r}\left(  \xi_{0}\right)  \right)  \right)  ^{q\times q}}%
\end{align*}
with%
\[
\omega\left(  r\right)  =\sup_{i,j=1,2,...,q}\left\vert \widetilde{a}%
_{ij}\right\vert _{C_{X}^{\alpha}\left(  \widetilde{B}_{R}\left(  \xi
_{9}\right)  \right)  }r^{\alpha}%
\]
Hence for $r$ small enough $\mathcal{T}$ is a contraction of $\left(
C_{X,\ast}^{\alpha}\left(  B_{r}\left(  \xi_{0}\right)  \right)  \right)
^{q\times q}$ in itself.
\end{proof}

Next, we need the following similar result:

\begin{theorem}
\label{thm contraction 1}If $v\in C_{X,0}^{2,\alpha}\left(  \widetilde
{B}\left(  \xi_{0},R\right)  \right)  $, $\widetilde{a}_{ij},\widetilde{L}v\in
C_{X}^{1,\alpha}\left(  \widetilde{B}_{R}\left(  \xi_{0}\right)  \right)  $
and $r$ is small enough, the operator $\mathcal{T}$ is also a contraction of
$\left(  C_{X,\ast}^{1,\alpha}\left(  \widetilde{B}_{r}\left(  \xi_{0}\right)
\right)  \right)  ^{q\times q}$ in itself.
\end{theorem}

\begin{proof}
We already know that%
\[
\beta G_{l,m}=\beta T_{lm}\left(  \xi_{0}\right)  \widetilde{\mathcal{L}%
}v+\sum_{k=1}^{q}\beta T_{lm,k}^{A}\left(  \xi_{0}\right)  \widetilde{X}%
_{k}v+\beta T_{lm}^{A}\left(  \xi_{0}\right)  v\in C_{X,0}^{\alpha}\left(
\widetilde{B}_{r}\left(  \xi_{0}\right)  \right)  .
\]
To show that also $\widetilde{X}_{h}\left(  \beta G_{l,m}\right)  \in
C_{X,\ast}^{\alpha}\left(  \widetilde{B}_{r}\left(  \xi_{0}\right)  \right)
$, let us compute%

\begin{align*}
\widetilde{X}_{h}\left(  \beta G_{l,m}\right)   &  =\left(  \widetilde{X}%
_{h}\beta\right)  G_{l,m}+\beta\left\{  \widetilde{X}_{h}T_{lm}\left(  \xi
_{0}\right)  \widetilde{\mathcal{L}}v+\right. \\
&  +\left.  \sum_{k=1}^{q}\widetilde{X}_{h}T_{lm,k}^{A}\left(  \xi_{0}\right)
\widetilde{X}_{k}v+\widetilde{X}_{h}T_{lm}^{A}\left(  \xi_{0}\right)
v\right\}
\end{align*}
exploiting Theorem \ref{thm type 0 exchange}%
\begin{align*}
&  =\left(  \widetilde{X}_{h}\beta\right)  G_{l,m}+\beta\left\{  \left(
\sum_{s=1}^{q}T_{lm}^{s}\left(  \xi_{0}\right)  \widetilde{X}_{s}+T_{lm}%
^{0}\left(  \xi_{0}\right)  \right)  \widetilde{\mathcal{L}}v+\right. \\
&  +\left.  \sum_{k=1}^{q}\left(  \sum_{s=1}^{q}T_{lm,k}^{A,s}\left(  \xi
_{0}\right)  \widetilde{X}_{s}+T_{lm,k}^{A,0}\left(  \xi_{0}\right)  \right)
\widetilde{X}_{k}v\right. \\
&  \left.  +\left(  \sum_{s=1}^{q}T_{lm}^{A,s}\left(  \xi_{0}\right)
\widetilde{X}_{s}+T_{lm}^{A,0}\left(  \xi_{0}\right)  \right)  v\right\}  .
\end{align*}
Recalling that $v\in C_{X,0}^{2,\alpha}\left(  \widetilde{B}_{R}\left(
\xi_{0}\right)  \right)  $ by (\ref{continuity}) we get that the quantity in
$\left\{  ...\right\}  $ belongs to $C_{X}^{\alpha}\left(  \widetilde{B}%
_{R}\left(  \xi_{0}\right)  \right)  $, hence by our choice of $\beta$,
\[
\widetilde{X}_{h}\left(  \beta G_{l,m}\right)  \in C_{X,0}^{\alpha}\left(
\widetilde{B}_{r}\left(  \xi_{0}\right)  \right)  \subset C_{X,\ast}^{\alpha
}\left(  \widetilde{B}_{r}\left(  \xi_{0}\right)  \right)  .
\]

As to the other term of $\mathcal{T}\left(  F\right)  $,
\[
\beta T_{lm}\left(  \xi_{0}\right)  \left(  \sum_{i,j=1}^{q}\left[
\widetilde{a}_{ij}(\xi_{0})-\widetilde{a}_{ij}(\cdot)\right]  \,F_{ij}\right)
,
\]
for $\widetilde{a}_{ij}\in C_{X}^{1,\alpha}\left(  \widetilde{B}_{R}\left(
\xi_{0}\right)  \right)  ,F_{ij}\in C_{X,\ast}^{1,\alpha}\left(  \widetilde
{B}_{r}\left(  \xi_{0}\right)  \right)  $ we can compute, by Theorem
\ref{thm type 0 exchange}:%
\begin{align*}
&  \widetilde{X}_{k}\left(  \beta T_{lm}\left(  \xi_{0}\right)  \left(
\sum_{i,j=1}^{q}\left[  \widetilde{a}_{ij}(\xi_{0})-\widetilde{a}_{ij}%
(\cdot)\right]  \,F_{ij}\right)  \right) \\
&  =\beta\left\{  \left(  \sum_{s=1}^{q}T_{lm}^{s}\left(  \xi_{0}\right)
\widetilde{X}_{s}+T_{lm}^{0}\left(  \xi_{0}\right)  \right)  \left(
\sum_{i,j=1}^{q}\left[  \widetilde{a}_{ij}(\xi_{0})-\widetilde{a}_{ij}%
(\cdot)\right]  F_{ij}\right)  \right\}  +
\end{align*}%
\begin{align*}
&  +\left(  \widetilde{X}_{k}\beta\right)  T_{lm}\left(  \xi_{0}\right)
\left(  \sum_{i,j=1}^{q}\left[  \widetilde{a}_{ij}(\xi_{0})-\widetilde{a}%
_{ij}(\cdot)\right]  \,F_{ij}\right) \\
&  =\beta\left\{  \sum_{s=1}^{q}T_{lm}^{s}\left(  \xi_{0}\right)  \left(
\sum_{i,j=1}^{q}\left[  \widetilde{a}_{ij}(\xi_{0})-\widetilde{a}_{ij}%
(\cdot)\right]  \widetilde{X}_{s}F_{ij}\right)  -\sum_{s=1}^{q}T_{lm}%
^{s}\left(  \xi_{0}\right)  \left(  \sum_{i,j=1}^{q}\left(  \widetilde{X}%
_{s}\widetilde{a}_{ij}\right)  F_{ij}\right) + \right. \\
&  +\left.  T_{lm}^{0}\left(  \xi_{0}\right)  \left(  \sum_{i,j=1}^{q}\left[
\widetilde{a}_{ij}(\xi_{0})-\widetilde{a}_{ij}(\cdot)\right]  F_{ij}\right)
\right\}  +\left(  \widetilde{X}_{k}\beta\right)  T_{lm}\left(  \xi
_{0}\right)  \left(  \sum_{i,j=1}^{q}\left[  \widetilde{a}_{ij}(\xi
_{0})-\widetilde{a}_{ij}(\cdot)\right]  \,F_{ij}\right)  .
\end{align*}
Since, under our assumptions, all the functions:%
\begin{align*}
&  \sum_{i,j=1}^{q}\left[  \widetilde{a}_{ij}(\xi_{0})-\widetilde{a}%
_{ij}(\cdot)\right]  \widetilde{X}_{s}F_{ij}\\
&  \sum_{i,j=1}^{q}\left(  \widetilde{X}_{s}\widetilde{a}_{ij}\right)
F_{ij}\\
&  \sum_{i,j=1}^{q}\left[  \widetilde{a}_{ij}(\xi_{0})-\widetilde{a}%
_{ij}(\cdot)\right]  F_{ij}%
\end{align*}
belong to $C_{X,\ast}^{\alpha}\left(  \widetilde{B}_{r}(\xi_{0})\right)
\subset C_{X,0}^{\alpha}\left(  \widetilde{B}_{R}(\xi_{0})\right)  ,$ by
(\ref{continuity}) and our choice of $\beta$ we conclude%
\[
\widetilde{X}_{k}\left(  \beta T_{lm}\left(  \xi_{0}\right)  \left(
\sum_{i,j=1}^{q}\left[  \widetilde{a}_{ij}(\xi_{0})-\widetilde{a}_{ij}%
(\cdot)\right]  \,F_{ij}\right)  \right)  \in C_{X,\ast}^{\alpha}\left(
\widetilde{B}_{r}(\xi_{0})\right)  ,
\]
hence $\mathcal{T}$ maps $C_{X,\ast}^{1,\alpha}\left(  \widetilde{B}%
_{r}\left(  \xi_{0}\right)  \right)  $ in itself. Let us show that
$\mathcal{T}$ is also a contraction in $C_{X,\ast}^{1,\alpha}\left(
\widetilde{B}_{r}\left(  \xi_{0}\right)  \right)  $. We already know that%
\begin{equation}
\left\Vert \mathcal{T}F^{\left(  1\right)  }-\mathcal{T}F^{\left(  2\right)
}\right\Vert _{C_{X}^{\alpha}\left(  \widetilde{B}_{r}\left(  \xi_{0}\right)
\right)  ^{q\times q}}\leq c\omega\left(  r\right)  \left\Vert F^{\left(
1\right)  }-F^{\left(  2\right)  }\right\Vert _{\left(  C_{X}^{\alpha}\left(
\widetilde{B}_{r}\left(  \xi_{0}\right)  \right)  \right)  ^{q\times q}}
\label{contraz 0}%
\end{equation}
so let us compute%
\begin{align*}
&  \widetilde{X}_{k}\mathcal{T}F^{\left(  1\right)  }-\widetilde{X}%
_{k}\mathcal{T}F^{\left(  2\right)  }\\
&  =\beta\left\{  \sum_{s=1}^{q}T_{lm}^{s}\left(  \xi_{0}\right)  \left(
\sum_{i,j=1}^{q}\left[  \widetilde{a}_{ij}(\xi_{0})-\widetilde{a}_{ij}%
(\cdot)\right]  \left(  \widetilde{X}_{s}F_{ij}^{\left(  1\right)
}-\widetilde{X}_{s}F_{ij}^{\left(  2\right)  }\right)  \right)  \right. \\
&  -\sum_{s=1}^{q}T_{lm}^{s}\left(  \xi_{0}\right)  \left(  \sum_{i,j=1}%
^{q}\widetilde{X}_{s}\widetilde{a}_{ij}\left(  F_{ij}^{\left(  1\right)
}-F_{ij}^{\left(  2\right)  }\right)  \right)  +
\end{align*}%
\begin{align*}
&  +\left.  T_{lm}^{0}\left(  \xi_{0}\right)  \left(  \sum_{i,j=1}^{q}\left[
\widetilde{a}_{ij}(\xi_{0})-\widetilde{a}_{ij}(\cdot)\right]  \left(
F_{ij}^{\left(  1\right)  }-F_{ij}^{\left(  2\right)  }\right)  \right)
\right\}  +\\
&  +\left(  \widetilde{X}_{k}\beta\right)  T_{lm}\left(  \xi_{0}\right)
\left(  \sum_{i,j=1}^{q}\left[  \widetilde{a}_{ij}(\xi_{0})-\widetilde{a}%
_{ij}(\cdot)\right]  \,\left(  F_{ij}^{\left(  1\right)  }-F_{ij}^{\left(
2\right)  }\right)  \right) \\
&  \equiv A+B+C+D.
\end{align*}
Applying again Theorem \ref{Thm type zero continuity}\ and (\ref{1.2}),%
\begin{align}
&  \left\Vert A\right\Vert _{C_{X}^{\alpha}\left(  \widetilde{B}_{r}\left(
\xi_{0}\right)  \right)  }\leq c\omega\left(  r\right)  \left\Vert
XF_{ij}^{\left(  1\right)  }-XF_{ij}^{\left(  2\right)  }\right\Vert
_{C_{X}^{\alpha}\left(  \widetilde{B}_{r}\left(  \xi_{0}\right)  \right)
}\label{A}\\
&  \leq c\omega\left(  r\right)  \left\Vert F_{ij}^{\left(  1\right)  }%
-F_{ij}^{\left(  2\right)  }\right\Vert _{C_{X}^{1,\alpha}\left(
\widetilde{B}_{r}\left(  \xi_{0}\right)  \right)  }\nonumber\\
&  \left\Vert C\right\Vert _{C_{X}^{\alpha}\left(  \widetilde{B}_{r}\left(
\xi_{0}\right)  \right)  }+\left\Vert D\right\Vert _{C_{X}^{\alpha}\left(
\widetilde{B}_{r}\left(  \xi_{0}\right)  \right)  }\leq c\omega\left(
r\right)  \left\Vert F_{ij}^{\left(  1\right)  }-F_{ij}^{\left(  2\right)
}\right\Vert _{C_{X}^{\alpha}\left(  \widetilde{B}_{r}\left(  \xi_{0}\right)
\right)  }\label{C, D}\\
&  \left\Vert B\right\Vert _{C_{X}^{\alpha}\left(  \widetilde{B}_{r}\left(
\xi_{0}\right)  \right)  }\leq c\sum_{s,i,j=1}^{q}\left\Vert X\widetilde
{a}_{ij}\right\Vert _{C_{X}^{\alpha}\left(  \widetilde{B}_{r}\left(  \xi
_{0}\right)  \right)  }\left\Vert F_{ij}^{\left(  1\right)  }-F_{ij}^{\left(
2\right)  }\right\Vert _{C_{X}^{\alpha}\left(  \widetilde{B}_{r}\left(
\xi_{0}\right)  \right)  }.\nonumber
\end{align}

To complete the bound on $B$, let us note that if $g\in C_{X,\ast}^{1,\alpha
}\left(  \widetilde{B}_{r}\left(  \xi_{0}\right)  \right)  $ we have%
\[
\left\Vert g\right\Vert _{\infty}\leq\sup_{\xi,\eta\in\widetilde{B}_{r}\left(
\xi_{0}\right)  }\left\vert g\left(  \xi\right)  -g\left(  \eta\right)
\right\vert \leq\left\vert g\right\vert _{C^{\alpha}\left(  \widetilde{B}%
_{r}\left(  \xi_{0}\right)  \right)  }\left(  2r\right)  ^{\alpha}%
\]
and applying (\ref{4.6}) (seeing $g$ as a function in $C_{X,0}^{1,\alpha
}\left(  \widetilde{B}_{R}\left(  \xi_{0}\right)  \right)  $),%
\[
\left\vert g\right\vert _{C^{\alpha}\left(  \widetilde{B}_{r}\left(  \xi
_{0}\right)  \right)  }=\sup_{\xi,\eta\in\widetilde{B}_{r}\left(  \xi
_{0}\right)  }\frac{\left\vert g\left(  \xi\right)  -g\left(  \eta\right)
\right\vert }{d_{\widetilde{X}}\left(  \xi,\eta\right)  ^{\alpha}}\leq
\sup_{\widetilde{B}_{r}\left(  \xi_{0}\right)  }\left\vert \widetilde
{X}g\right\vert \left(  2r\right)  ^{1-\alpha},
\]
hence%
\[
\left\Vert g\right\Vert _{C_{X}^{\alpha}\left(  \widetilde{B}_{r}\left(
\xi_{0}\right)  \right)  }\leq\left\Vert g\right\Vert _{C_{X}^{1,\alpha
}\left(  \widetilde{B}_{r}\left(  \xi_{0}\right)  \right)  }\left(  \left(
2r\right)  ^{\alpha}+\left(  2r\right)  ^{1-\alpha}\right)
\]
and%
\begin{equation}
\left\Vert B\right\Vert _{C_{X}^{\alpha}\left(  \widetilde{B}_{r}\left(
\xi_{0}\right)  \right)  }\leq c\sum_{s,i,j=1}^{q}\left\Vert \widetilde{X}%
_{s}\widetilde{a}_{ij}\right\Vert _{C_{X}^{\alpha}\left(  B_{r}\left(  \xi
_{0}\right)  \right)  }\left(  \left(  2r\right)  ^{\alpha}+\left(  2r\right)
^{1-\alpha}\right)  \left\Vert F_{ij}^{\left(  1\right)  }-F_{ij}^{\left(
2\right)  }\right\Vert _{C_{X}^{1,\alpha}\left(  \widetilde{B}_{r}\left(
\xi_{0}\right)  \right)  }. \label{B}%
\end{equation}
From (\ref{contraz 0}), (\ref{A}), (\ref{C, D}), (\ref{B}) we deduce that for
$r$ small enough%
\[
\left\Vert \mathcal{T}F^{\left(  1\right)  }-\mathcal{T}F^{\left(  2\right)
}\right\Vert _{\left(  C_{X}^{1,\alpha}\left(  \widetilde{B}_{r}\left(
\xi_{0}\right)  \right)  \right)  ^{q\times q}}\leq\delta\left\Vert F^{\left(
1\right)  }-F^{\left(  2\right)  }\right\Vert _{\left(  C_{X}^{1,\alpha
}\left(  B_{r}\left(  \xi_{0}\right)  \right)  \right)  ^{q\times q}}%
\]
with $\delta<1,$ and we are done.
\end{proof}

We now come to the

\bigskip

\begin{proof}
[Conclusion of the proof of Theorem \ref{thm main linear}]By Remark
\ref{remark principal part} it is enough to prove the theorem for $b_{i}=c=0$.
We will prove the regularity result for $k=1$; an iterative argument gives the
general case. Also, once the $C_{X,loc}^{k+2,\alpha}\left(  \Omega\right)  $
is proved for every $k$, H\"{o}rmander's condition implies that a solution
$u\in C_{X,loc}^{k+2,\alpha}\left(  \Omega\right)  $ for any $k$ is also
smooth in Euclidean sense.

So, let $u\in C_{X}^{2,\alpha}\left(  \Omega\right)  $ satisfy the equation%
\[
Lu\equiv\sum_{i,j=1}^{q}a_{ij}\left(  x\right)  X_{i}X_{j}u=f\text{ in }\Omega
\]
and assume that%
\[
a_{ij},f,\in C_{X}^{1,\alpha}\left(  \Omega\right)  .
\]
Fix $x_{0}\in\Omega$ and a small ball $B_{R}\left(  x_{0}\right)
\subset\Omega$ where the lifting procedure is applicable. Let $\xi_{0}=\left(
x_{0},0\right)  ,\xi=\left(  x,h\right)  $. Then by Proposition
\ref{Proposition de-lifting},
\begin{align*}
\widetilde{u}\left(  \xi\right)   &  =u\left(  x\right)  \in C_{X}^{2,\alpha
}\left(  \widetilde{B}_{R}\left(  \xi_{0}\right)  \right)  ,\\
\widetilde{a}_{ij}\left(  \xi\right)   &  =a_{ij}\left(  x\right)
,\widetilde{f}\left(  \xi\right)  =f\left(  x\right)  \in C_{X}^{1,\alpha
}\left(  \widetilde{B}_{R}\left(  \xi_{0}\right)  \right) \\
\widetilde{L}\widetilde{u}  &  =\widetilde{f}\text{ in }\widetilde{B}%
_{R}\left(  \xi_{0}\right)  .
\end{align*}
Then, let $\phi\in C_{0}^{\infty}\left(  \widetilde{B}_{r}\left(  \xi
_{0}\right)  \right)  $ such that $\phi=1$ in $\widetilde{B}_{r/2}\left(
\xi_{0}\right)  $, hence $v=\phi\widetilde{u}\in C_{X,0}^{2,\alpha}\left(
\widetilde{B}_{R}\left(  \xi_{0}\right)  \right)  $ and%
\[
\widetilde{L}v=\phi\widetilde{f}+2\sum_{i,j=1}^{q}\widetilde{a}_{ij}%
\widetilde{X}_{i}\widetilde{u}\widetilde{X}_{j}\phi+\widetilde{u}\widetilde
{L}\phi\equiv g\in C_{X}^{1,\alpha}\left(  \widetilde{B}_{R}\left(  \xi
_{0}\right)  \right)  .
\]

For this function $v$ the representation formula (\ref{rep formula}) holds
true, with $G_{lm}$ given by (\ref{G_l,m}). Let us define the number $r$, the
cutoff function $\beta$ and the operator $\mathcal{T}$ as in Theorems
\ref{thm contraction 0}, \ref{thm contraction 1}. Since $\left(  C_{X,\ast
}^{\alpha}\left(  \widetilde{B}_{r}\left(  \xi_{0}\right)  \right)  \right)
^{q\times q}$ and $\left(  C_{X,\ast}^{1,\alpha}\left(  \widetilde{B}%
_{r}\left(  \xi_{0}\right)  \right)  \right)  ^{q\times q}$ are Banach spaces,
by the Banach-Caccioppoli Theorem the operator $\mathcal{T}$ possesses a
unique fixed point $W$ both in $\left(  C_{X,\ast}^{\alpha}\left(
\widetilde{B}_{r}\left(  \xi_{0}\right)  \right)  \right)  ^{q\times q}$ and
in $\left(  C_{X,\ast}^{1,\alpha}\left(  \widetilde{B}_{r}\left(  \xi
_{0}\right)  \right)  \right)  ^{q\times q}$. On the other hand, since
$\widetilde{X}_{i}\widetilde{X}_{j}v\in C_{X,\ast}^{\alpha}\left(
\widetilde{B}_{r}\left(  \xi_{0}\right)  \right)  $, by (\ref{rep formula})
choosing $a\left(  \xi\right)  =1$ in $\widetilde{B}_{R/2}\left(  \xi
_{0}\right)  \supset\widetilde{B}_{r/2}\left(  \xi_{0}\right)  $, we get
\[
\widetilde{X}_{i}\widetilde{X}_{j}v\left(  \xi\right)  =W\left(  \xi\right)
\text{ in }\widetilde{B}_{r/2}\left(  \xi_{0}\right)  ,
\]
hence $\widetilde{X}_{i}\widetilde{X}_{j}v\in C_{X}^{1,\alpha}\left(
\widetilde{B}_{r/2}\left(  \xi_{0}\right)  \right)  $ and also $\widetilde
{X}_{i}\widetilde{X}_{j}\widetilde{u}\in C_{X}^{1,\alpha}\left(  \widetilde
{B}_{r/2}\left(  \xi_{0}\right)  \right)  $. By Proposition
\ref{Proposition de-lifting} this implies%
\[
X_{i}X_{j}u\in C_{X}^{1,\alpha}\left(  B_{r/4}\left(  x_{0}\right)  \right)
\]
therefore $u\in C_{X}^{3,\alpha}\left(  B_{r/4}\left(  x_{0}\right)  \right)
$ and by a covering argument $u\in C_{X,loc}^{3,\alpha}\left(  \Omega\right)
$.
\end{proof}

\section{Smoothness of solutions to quasilinear
equations\label{sec quasilinear}}

Let us apply the previous linear theory to a regularity result for solutions
to quasilinear equations.

\begin{theorem}
\label{Thm main quasilinear}Let%
\[
Qu\equiv\sum_{i,j=1}^{q}a_{ij}\left(  x,u,Xu\right)  X_{i}X_{j}u+b\left(
x,u,Xu\right)
\]
where $X_{1},X_{2},...,X_{q}$ are as above, $a_{ij}=a_{ji},$
\[
\Lambda|\xi|^{2}\leq\sum_{i,j=1}^{q}a_{ij}\left(  x,u,p\right)  \xi_{i}\xi
_{j}\leq\Lambda^{-1}|\xi|^{2}\text{ }\forall\xi\in{\mathbb{R}}^{q},\left(
x,u,p\right)  \in\Omega\times\mathbb{R}\times\mathbb{R}^{q}.
\]
and assume that for some $k=1,2,3,...,\alpha\in(0,1)$%
\[
a_{ij},b\in C_{X}^{k,\alpha}\left(  \Omega\times\mathbb{R}\times\mathbb{R}%
^{q}\right)  .
\]
If $u\in C_{X}^{2,\alpha}\left(  \Omega\right)  $ solves the equation $Qu=0$,
then $u\in C_{X,loc}^{k+2,\alpha}\left(  \Omega\right)  $. In particular, if
$a_{ij},b$ are smooth, then $u$ is also smooth.
\end{theorem}

\begin{proof}
Under our assumptions we have that $u$ is a solution to the linear equation%
\[
Lf\left(  x\right)  \equiv\sum_{i,j=1}^{q}\overline{a}_{ij}\left(  x\right)
X_{i}X_{j}f\left(  x\right)  =g\left(  x\right)
\]
where
\begin{align*}
\overline{a}_{ij}\left(  x\right)   &  =a_{ij}\left(  x,u\left(  x\right)
,Xu\left(  x\right)  \right)  \in C_{X}^{1,\alpha}\left(  \Omega\right) \\
g\left(  x\right)   &  =-b\left(  x,u\left(  x\right)  ,Xu\left(  x\right)
\right)  \in C_{X}^{1,\alpha}\left(  \Omega\right)
\end{align*}
hence by Theorem \ref{thm main linear}, $u\in C_{X,loc}^{3,\alpha}\left(
\Omega\right)  $; if $k=1$ we are done, while if $k\geq2,$
\[
\text{since }u\in C_{X,loc}^{3,\alpha}\left(  \Omega\right)  \text{, then
}\overline{a}_{ij},g\in C_{X,loc}^{2,\alpha}\left(  \Omega\right)
\]
and by Theorem \ref{thm main linear} $u\in C_{X,loc}^{4,\alpha}\left(
\Omega\right)  $. Iteration gives the desired result. Again, H\"{o}rmander's
condition assures that if $u$ belongs to $C_{X,loc}^{k+2,\alpha}\left(
\Omega\right)  $ for any $k$, then it is also smooth in the Euclidean sense.
\end{proof}

\section{The evolution case\label{sec evolution}}

In virtue of the results contained in \cite{BB2007} all the previous theory
can be developed also in the evolution case. Let us consider the linear
operator%
\[
H=\partial_{t}-\sum_{i,j=1}^{q}a_{ij}\left(  t,x\right)  X_{i}X_{j}+\sum
_{i=1}^{q}b_{i}\left(  t,x\right)  X_{i}+c\left(  t,x\right)
\]
where the $X_{i}$'s are still a system of H\"{o}rmander's vector fields in a
neighborhood $\Omega_{0}$ of a bounded domain $\Omega$, $Q=\left(  0,T\right)
\times\Omega$. We define in $Q$ the parabolic distance%
\[
d_{P}\left(  \left(  t,x\right)  ,\left(  s,y\right)  \right)  =\sqrt
{d_{X}\left(  x,y\right)  ^{2}+\left\vert t-s\right\vert }%
\]
and define the spaces $C_{P}^{\alpha}\left(  Q\right)  $ of H\"{o}lder
continuous functions of exponent $\alpha$ with respect to the distance $d_{P}%
$, and the spaces $C_{P}^{k,\alpha}\left(  Q\right)  $ of functions such that
all the derivatives up to weight $k$ with respect to the $X_{i}$'s and
$\partial_{t}$, with $\partial_{t}$ weighting as a second order derivative,
belong to $C_{P}^{\alpha}\left(  Q\right)  $. We assume with $a_{ij}%
,b_{i},c\in C_{P}^{\alpha}\left(  Q\right)  $ for some $\alpha\in\left(
0,1\right)  $, $a_{ij}=a_{ji}$ satisfying the condition%
\[
\Lambda|\xi|^{2}\leq\sum_{i,j=1}^{q}a_{ij}\left(  t,x\right)  \xi_{i}\xi
_{j}\leq\Lambda^{-1}|\xi|^{2}\text{ }\forall\xi\in{\mathbb{R}}^{q},\left(
t,x\right)  \in Q.
\]
Then the same reasoning of \S \ref{sec linear} gives the following:

\begin{theorem}
\label{Thm linear evolution}Under the above assumptions, let $u\in
C_{P}^{2,\alpha}\left(  Q\right)  $ satisfy the equation%
\[
Hu=f\text{ in }Q
\]
and assume that for some integer $k=1,2,3,...$we have:%
\[
a_{ij},b_{i},c,f,\in C_{P}^{k,\alpha}\left(  Q\right)  .
\]
Then%
\[
u\in C_{P,loc}^{k+2,\alpha}\left(  Q\right)  .
\]

\end{theorem}

In particular, in this case the a priori estimates proved in \cite{BB2007}
apply:
\[
\left\Vert u\right\Vert _{C_{P}^{k+2,\alpha}\left(  Q^{\prime}\right)
}\leqslant c\left\{  \left\Vert Hu\right\Vert _{C_{P}^{k,\alpha}\left(
Q\right)  }+\left\Vert u\right\Vert _{L^{\infty}\left(  Q\right)  }\right\}
\]
for any $Q^{\prime}\Subset Q$, with constant $c$ independent of $u$.

We also get the following quasilinear counterpart:

\begin{theorem}
\label{Thm evolution quasilinear}Let%
\[
\mathcal{Q}u\equiv\partial_{t}u-\sum_{i,j=1}^{q}a_{ij}\left(  t,x,u,Xu\right)
X_{i}X_{j}u+b\left(  t,x,u,Xu\right)
\]
where $X_{1},X_{2},...,X_{q}$ are as above, $a_{ij}=a_{ji},$
\[
\Lambda|\xi|^{2}\leq\sum_{i,j=1}^{q}a_{ij}\left(  t,x,u,p\right)  \xi_{i}%
\xi_{j}\leq\Lambda^{-1}|\xi|^{2}%
\]
$\forall\xi\in{\mathbb{R}}^{q},\left(  t,x,u,p\right)  \in\left(  0,T\right)
\times\Omega\times\mathbb{R}\times\mathbb{R}^{q}$, and assume that for some
$k=1,2,3,...,\alpha\in(0,1)$%
\[
a_{ij},b\in C_{P}^{k,\alpha}\left(  \left(  0,T\right)  \times\Omega
\times\mathbb{R}\times\mathbb{R}^{q}\right)  .
\]
If $u\in C_{P}^{2,\alpha}\left(  Q\right)  $ solves the equation
$\mathcal{Q}u=0$, then $u\in C_{P,loc}^{k+2,\alpha}\left(  Q\right)  $. In
particular, if $a_{ij},b$ are smooth, then $u$ is also smooth.
\end{theorem}

\bigskip

Marco Bramanti

Dipartimento di Matematica

Politecnico di Milano

Via Bonardi 9

20133 Milano, ITALY

marco.bramanti@polimi.it

\bigskip

Maria Stella Fanciullo

Dipartimento di Matematica e Informatica

Universit\`{a} di Catania

Viale Andrea Doria 6

95125 Catania, ITALY

fanciullo@dmi.unict.it

\end{document}